\documentclass[a4paper,12pt]{article}
\usepackage[utf8]{inputenc}
\usepackage[english]{babel}
\usepackage{sectsty}
\usepackage{mathtools,amssymb, amsthm, amsfonts, stmaryrd, mathpazo}
\usepackage[backend=biber]{biblatex}
\usepackage{graphicx, subfig}
\usepackage{hhline}
\usepackage{comment}
\usepackage{rotating}
\usepackage{geometry}
\usepackage{titlesec}
\usepackage[absolute]{textpos}
\usepackage{tikz, tikz-qtree, tikz-cd}
\tikzset{
    lablv/.style={anchor=south, rotate=90, inner sep=.5mm},
    labld/.style={anchor=south, rotate=150, inner sep=.8mm},
    lablad/.style={anchor=north, rotate=40, inner sep=1.2mm}
}
\usepackage{pgffor,etoolbox}
\usepackage{pgfplotstable}
\usepackage{longtable}
\usepackage{booktabs}
\pgfplotsset{compat=1.18}
\usepackage{fp}
\usepackage{braket}
\usepackage{xstring}
\usepackage{youngtab ,ytableau}
\usepackage{enumitem}
\usepackage{fancyhdr}
\usepackage{colortbl}
\usepackage{color}
\definecolor{valentia}{rgb}{233,78,82}
\definecolor{titleblue}{rgb}{114, 146, 162}
\usepackage{array}
\usepackage{makecell}
\usepackage{float}
\usepackage{phoenician}
\usepackage{mdframed}
\usepackage{hyperref}
\usepackage{upgreek}
\usepackage{authblk}
\usepackage{multirow}
\usepackage{multicol}
\usepackage{scrextend}
\DeclareFontFamily{U}{mathb}{}
\DeclareFontShape{U}{mathb}{m}{n}{<-5.5> mathb5 <5.5-6.5> mathb6
<6.5-7.5> mathb7 <7.5-8.5> mathb8 <8.5-9.5> mathb9 <9.5-11> mathb10
<11-> mathb12}{}
\DeclareSymbolFont{mathb}{U}{mathb}{m}{n}
\DeclareFontSubstitution{U}{mathb}{m}{n}
\DeclareMathSymbol{\blackdiamond}{\mathbin}{mathb}{"0C}

\DeclareMathAlphabet{\mathscr}{U}{rsfso}{m}{n}
\DeclareMathSymbol{\smin}{\mathbin}{AMSa}{"39}
\usepackage{authblk}
\newcommand{\floor}[1]{\left\lfloor #1 \right\rfloor}
\newcommand{\Addresses}{{
  \bigskip
  \footnotesize

  \textsc{Institut Montpelliérain Alexander Grothendieck, Université Montpellier 2,  Place Eugène Bataillon, 34095 MONTPELLIER Cedex, FRANCE.}\par\nopagebreak
  \textit{E-mail address}: \texttt{raphael.paegelow@umontpellier.fr}
}}

\title{Combinatorics of the irreducible components of $\Hk_n^{\Gamma}$ in type $D$ and $E$}

\author{Raphaël Paegelow}
\begin{document}
\newgeometry{top=2cm, bottom=2.5cm, left=2.5cm, right=2.5cm}
\theoremstyle{definition}
\newtheorem{deff}{Definition}[section]
\newtheorem{nota}[deff]{Notation}

\newtheoremstyle{cremark}
    {\dimexpr\topsep/2\relax}
    {\dimexpr\topsep}
    {}
    {}
    {}
    {.}
    {.5em}
    {}

\theoremstyle{cremark}
\newtheorem{rmq}[deff]{Remark}
\newtheorem{ex}[deff]{Example}

\theoremstyle{plain}

\newtheorem{thm}[deff]{Theorem}
\newtheorem{cor}[deff]{Corollary}
\newtheorem{prop}[deff]{Proposition}
\newtheorem{lemme}[deff]{Lemma}
\newtheorem{fait}[deff]{Fact}
\newtheorem{conj}[deff]{Conjecture}
\newtheorem*{theorem1}{Theorem 1}
\newtheorem*{theorem2}{Theorem 2}
\newtheorem*{theorem3}{Theorem 3}
\newtheorem*{cor1}{Corollary}
\binoppenalty=10000
\relpenalty=10000


\newcommand{\TODO}{\colorbox{red}{\textbf{{TODO}}}}

\newcommand{\DMQ}{\overline{Q^0_{\Gamma}}}
\newcommand{\DMQGf}{\overline{Q_{G^f}}}
\newcommand{\DMQf}{\overline{Q_{\Gamma^f}}}
\newcommand{\DMQJf}{\overline{Q_{\bullet^f}}}

\newcommand{\QVG}{\mathcal{M}^{\Gamma}}
\newcommand{\tQVG}{\tilde{\mathcal{M}}^{\Gamma}}
\newcommand{\QV}{\mathcal{M}}
\newcommand{\tQV}{\tilde{\mathcal{M}}}
\newcommand{\QVTM}{\mathcal{M}^{\Gamma}_{\boldsymbol{\theta}}(M)}
\newcommand{\QVT}{\mathcal{M}^{\Gamma}_{\boldsymbol{\theta}}}
\newcommand{\QVTdef}{\mathcal{M}^{\Gamma}_{\boldsymbol{\theta},\lb}}
\newcommand{\tQVTdef}{\tilde{\mathcal{M}}^{\Gamma}_{\boldsymbol{\theta}, \lb}}
\newcommand{\tQVT}{\tilde{\mathcal{M}}^{\Gamma}_{\boldsymbol{\theta}}}
\newcommand{\tQVTM}{\tilde{\mathcal{M}}^{\Gamma}_{\boldsymbol{\theta}}(M)}
\newcommand{\QVTMs}{\mathcal{M}^{\Gamma}_{\boldsymbol{\theta}}(M)}
\newcommand{\tQVTMs}{\tilde{\mathcal{M}}^{\Gamma}_{\boldsymbol{\theta}}(M)}
\newcommand{\tQVTMb}{\tilde{\mathcal{M}}^{\Gamma}_{\boldsymbol{\theta},\lb\delta^{\Gamma}}(M)}
\newcommand{\QVTMb}{\mathcal{M}^{\Gamma}_{\boldsymbol{\theta},\lb\delta^{\Gamma}}(M)}
\newcommand{\QVTMdef}{\mathcal{M}^{\Gamma}_{\lb\delta^{\Gamma}}(M)}
\newcommand{\tQVTMdef}{\tilde{\mathcal{M}}^{\Gamma}_{\lb\delta^{\Gamma}}(M)}
\newcommand{\QVTd}{\mathcal{M}^{\Gamma}_{\boldsymbol{\theta}}(d)}
\newcommand{\tQVTd}{\tilde{\mathcal{M}}^{\Gamma}_{\boldsymbol{\theta}}(d)}
\newcommand{\QVJn}{\mathcal{M}^{\bullet}_{\boldsymbol{1}}(n)}
\newcommand{\tQVJn}{\tilde{\mathcal{M}}^{\bullet}_{\boldsymbol{1}}(n)}
\newcommand{\QVJnG}{{\mathcal{M}^{\bullet}_{\boldsymbol{1}}(n)}^{\Gamma}}
\newcommand{\QVJndefu}{{\mathcal{M}^{\bullet}_{1}(n)}}
\newcommand{\tQVJndefu}{\tilde{\mathcal{M}}^{\bullet}_{1}(n)}
\newcommand{\QVJnGdef}{{\mathcal{M}^{\bullet}_{1}(n)}^{\Gamma}}
\newcommand{\tQVJndef}{\tilde{\mathcal{M}}^{\bullet}_{\lb}(n)}
\newcommand{\tQVJns}{\tilde{\mathcal{M}}^{\bullet}_{\boldsymbol{\theta}}(n)}
\newcommand{\QVJns}{{\mathcal{M}^{\bullet}_{\boldsymbol{\theta}}(n)}}
\newcommand{\QVJndef}{{\mathcal{M}^{\bullet}_{\lb}(n)}}
\newcommand{\tQVJnb}{\tilde{\mathcal{M}}^{\bullet}_{\boldsymbol{\theta},\lb}(n)}
\newcommand{\QVJnb}{{\mathcal{M}^{\bullet}_{\boldsymbol{\theta},\lb}(n)}}

\newcommand{\Hom}{\mathrm{Hom}}
\newcommand{\GL}{\mathrm{GL}}
\newcommand{\Tr}{\mathrm{Tr}}
\newcommand{\Aut}{\mathrm{Aut}}
\newcommand{\mult}{\mathrm{mult}}
\newcommand{\module}{\mathrm{mod}}
\newcommand{\Dyn}{\mathrm{Dyn}}
\newcommand{\End}{\mathrm{End}}
\newcommand{\SL}{\mathrm{SL}}
\newcommand{\SU}{\mathrm{SU}}
\newcommand{\G}{\mathrm{G}}
\newcommand{\op}{\mathrm{op}}
\newcommand{\RepGa}{\mathbf{McK}(\Gamma)}
\newcommand{\Fr}{\mathrm{Fr}}
\newcommand{\Reg}{\mathrm{Reg}}
\newcommand{\id}{\mathrm{id}}
\newcommand{\Res}{\mathrm{Res}}
\newcommand{\std}{\mathrm{std}}
\newcommand{\Chi}{\mathrm{X}}
\newcommand{\Ind}{\mathrm{Ind}}
\newcommand{\Repm}{\mathbf{Rep}(\overline{Q_{\Gamma^f}})}

\newcommand{\Hk}{\mathcal{H}}
\newcommand{\Ck}{\mathscr{X}}
\newcommand{\Xk}{\mathcal{X}}

\newcommand{\Rep}{\mathrm{Rep}}
\newcommand{\RepG}{\mathrm{Rep}_{\Gamma}}
\newcommand{\RepQ}{\mathcal{R}}
\newcommand{\RepGn}{\mathrm{Rep}_{\Gamma,n}}
\newcommand{\RepGk}{\mathrm{Rep}_{\Gamma,k}}
\newcommand{\CharG}{\mathrm{Char}_{\Gamma}}
\newcommand{\CharGn}{\mathrm{Char}_{\Gamma,n}}
\newcommand{\CharGk}{\mathrm{Char}_{\Gamma,k}}
\newcommand{\Xstd}{X_{\mathrm{std}}}

\newcommand{\BD}{\mathrm{BD}}
\newcommand{\BT}{\mathrm{BT}}
\newcommand{\Gamman}{\Gamma_n}
\newcommand{\Dt}{\tilde{D}}
\newcommand{\LDt}{L(\Dt_{l+2})}
\newcommand{\Dtc}{\tilde{D}}
\newcommand{\Ct}{\tilde{C}}
\newcommand{\Dtl}{\tilde{D}_{l+2}}
\newcommand{\WaD}{W({\tilde{D}_{l+2}})}
\newcommand{\WD}{W({D_{l+2})}}
\newcommand{\IG}{\mathrm{Irr}_{\Gamma}}

\newcommand{\Qc}{Q^{\vee}}
\newcommand{\QDts}{Q(\Dt_{\ell+2}^{\sigma})[0^+]}
\newcommand{\deltD}{\delta^{\BD_{2\ell}}}
\newcommand{\av}{\alpha^{\vee}}
\newcommand{\ABDz}{\mathcal{A}^{n,\T_1}_{\BD_{2\ell}}}

\newcommand{\Wa}{W_{\Gamma}^{\text{aff}}}

\newcommand{\Taut}{\mathcal{T}_{d}}
\newcommand{\TautT}{\tilde{\mathcal{T}}_{d}}

\newcommand{\bigzero}{\mbox{\normalfont\Large\bfseries 0}}
\newcommand{\rvline}{\hspace*{-\arraycolsep}\vline\hspace*{-\arraycolsep}}
\newcommand{\C}{\mathbb{C}}
\newcommand{\Z}{\mathbb{Z}}
\newcommand{\R}{\mathbb{R}}
\newcommand{\Q}{\mathbb{Q}}
\newcommand{\lb}{\lambda}
\newcommand{\iso}{\xrightarrow{\,\smash{\raisebox{-0.65ex}{\ensuremath{\scriptstyle\sim}}}\,}}
\newcommand{\Sk}{\mathfrak{S}}
\newcommand{\Uk}{\mathcal{U}}
\newcommand{\T}{\mathbb{T}}
\newcommand{\Pro}{\mathscr{P}}
\newcommand{\Resol}{\mathscr{R}}
\newcommand{\Pnlo}{\mathcal{P}_{n,l}^{o}}
\newcommand{\IC}{\mathcal{IC}}
\newcommand{\gl}{\mathrm{g}_l}
\newcommand{\gG}{\mathrm{g}_{\Gamma}}
\newcommand{\gk}{\mathrm{g}_k}
\newcommand{\gll}{\mathrm{g}_{2\ell}}
\newcommand{\Cln}{C_{l,n}}
\newcommand{\Ctwo}{{(\C^2)}}
\newcommand{\oando}{\boldsymbol{\theta},\lb}
\newcommand{\Ch}{\mathfrak{C}}
\newcommand{\Alf}{\mathfrak{A}}

\newcommand{\qphoe}{\textphnc{q}}
\newcommand{\lphoe}{\textphnc{l}}
\newcommand{\Sphoe}{\textphnc{\As}}
\newcommand{\hethphoe}{\textphnc{\Ahd}}


\DeclareRobustCommand*{\Sagelogo}{%
  \begin{tikzpicture}[line width=.2ex,line cap=round,rounded corners=.01ex,baseline=-.2ex]
    \draw(0,0) -- (.75em,0)
      -- (.75em,.7ex) -- (.25em,.7ex)
      -- (.25em,.75\ht\strutbox) -- (.75em,.75\ht\strutbox);
    \draw(2em,0) -- (1.6em,0)
      -- (1.3em,.75\ht\strutbox) -- (.9em,.75\ht\strutbox)
      -- (.9em,0) -- (1.3em,0)
      -- (1.6em,.75\ht\strutbox) -- (2.1em,.75\ht\strutbox)
      -- (2.1em,-\dp\strutbox) -- (1.45em,-\dp\strutbox);
    \draw(3em,0) -- (2.25em,0)
      -- (2.25em,.75\ht\strutbox) -- (2.8em,.75\ht\strutbox)
      -- (2.8em,.7ex) -- (2.35em, .7ex);
  \end{tikzpicture}%
}

\maketitle
\section*{Abstract}
In this article, we give a combinatorial model (in terms of symmetric cores) of the indexing set of the irreducible components of $\Hk_n^{\Gamma}$ (the $\Gamma$-fixed points of the Hilbert scheme of $n$ points in the plane) containing a monomial ideal, whenever $\Gamma$ is a finite subgroup of $\SL_2(\C)$ isomorphic to the binary dihedral group. Moreover, we show that if $\Gamma$ is a subgroup of $\SL_2(\C)$ isomorphic to the binary tetrahedral group, to the binary octahedral group or to the binary icosahedral group, then the $\Gamma$-fixed points of $\Hk_n$ which are also fixed under the maximal diagonal torus of $\SL_2(\C)$, are in fact $\SL_2(\C)$-fixed points. Finally, we prove that in this case, the irreducible components of $\Hk_n^{\Gamma}$ containing a monomial ideal are zero-dimensional.

\section{Introduction}
Let $\Gamma$ be a finite subgroup of $\SL_2(\C)$ and for $n \in \Z_{\geq 0}$, let $\Hk_n$ be the Hilbert scheme of $n$ points in $\C^2$. The natural action of $\Gamma$ on $\C^2$, induces a $\Gamma$-action on $\C[x,y]$ and thus on $\Hk_n$. In this article, we are interested in the combinatorics of the parametrization set of the irreducible components of  $\Hk_n^{\Gamma}$.
 When $\Gamma$ is equal to the cyclic subgroup of the maximal diagonal torus of $\SL_2(\C)$, a combinatorial model using partitions has already been constructed by Iain Gordon \cite[Lemma $7.8$]{Gor08} and by Cédric Bonnafé and Ruslan Maksimau \cite[Lemma $4.9$]{BM21}. We will therefore only consider the groups of type $D$ and $E$. Type $D$ corresponds to the class of finite subgroups of $\SL_2(\C)$ that are isomorphic to the binary dihedral subgroups. In the second section, we introduce important notation concerning affine root systems and partitions of integers. In the third section, we then define the binary dihedral group and give its character table and its McKay graph. We then present a folding of that Dynkin diagram which will be of use in the next section. In section four, we define and give the main properties of a generalisation of the residue to type $D$. In the fifth section, we prove the first theorem, which can be stated as follows.

\begin{theorem1}
Let $\ell$ be an integer greater or equal to $2$ and $\Gamma$ be a binary dihedral subgroup of $\SL_2(\C)$ of order $4\ell$. Then the set of all irreducible components of $\Hk_n^{\Gamma}$ containing a monomial ideal is in bijection with the set of all symmetric $2\ell$-cores $\lambda$, such that $|\lb|\equiv n\text{ } [2\ell]$ and $|\lb| \leq n$. Moreover, for each $\mu_1,\mu_2$ symmetric partitions of $n$, the monomial ideals attached to ${\mu_1}$ and ${\mu_2}$ are in the same irreducible component of $\Hk_n^{\Gamma}$ if and only if the $2\ell$-cores of $\mu_1$ and $\mu_2$ are equal.

\end{theorem1}

\noindent In section six, we start by giving a presentation of the binary tetrahedral group, its character table and its McKay graph. Moreover, we prove that if $\Gamma$ is isomorphic to the binary tetrahedral group, then the points in $\Hk_n$ that are fixed under $\Gamma$ and the maximal diagonal torus of $\SL_2(\C)$ are exactly the $\SL_2(\C)$-fixed points. Since the binary octahedral group and the binary icosahedral group contain a subgroup isomorphic to the binary tetrahedral group, the previous result generalises to these two isomorphism classes of finite subgroups of $\SL_2(\C)$. Finally, in section seven, we prove the following theorem.

\begin{theorem2}
If $\Gamma$ is a finite subgroup of $\SL_2(\C)$ of type $E$, then for each $I \in \Hk_n^{\SL_2(\C)}$, the irreducible component of $\Hk_n^{\Gamma}$ containing $I$ is zero-dimensional.
\end{theorem2}

\subsection{Acknowledgement}

The author would like to thank C\'edric Bonnaf\'e for suggesting the study of the combinatorics of the irreducible components of the $\Gamma$-fixed points in the punctual Hilbert scheme of $\C^2$, when $\Gamma$ is of type $D$ and $E$ and the referee for valuable comments, suggestions and improvements.

\section{Starting point}
Fix $\Gamma$ a finite subgroup of $\SL_2(\C)$. In this subsection we recall the general description of the indexing set of the irreducible components of $\Hk_n^{\Gamma}$, in terms of roots, that has been obtained in \cite{Pae1}. Denote by $I_{\Gamma}$ the set of all irreducible characters of $\Gamma$ and let $\chi_0 \in I_{\Gamma}$ denote the trivial character of $\Gamma$. Let respectively $\Delta_{\Gamma}^+ (\subset \Delta_{\Gamma})$ be the free monoid (free abelian group) associated with $I_{\Gamma}$. Let $\tilde{T}_{\Gamma}$ be the type of the McKay graph seen as an affine Dynkin diagram. One can then associate with $\Gamma$ a realization $\left(\mathfrak{h}_{\Gamma},\Pi_{\Gamma}:=\left\{\alpha_{\chi} | \chi \in I_{\Gamma}\right\},\Pi_{\Gamma}^{\vee}:=\left\{\alpha^{\vee}_{\chi} | \chi \in I_{\Gamma}\right\}\right)$ \cite[§1.1]{kac90} of the generalized Cartan matrix of type $\tilde{T}_{\Gamma}$. Denote respectively by $Q(\tilde{T}_{\Gamma})$ and $W(\tilde{T}_{\Gamma})$ the root lattice and Weyl group associated with the previously mentioned realization. From now on, we will identify $Q(\tilde{T}_{\Gamma})$ with $\Delta_{\Gamma}$. Let $\delta^{\Gamma}$ denote the null root.\\
For $d \in \Delta_{\Gamma}$, let $(d_{\chi}) \in \Z^{|I_{\Gamma}|}$ be such that $d=\sum_{\chi \in I_{\Gamma}}{d_{\chi} \chi}$. For each $d \in \Delta^+_{\Gamma}$, let $|d|_{\Gamma}:=\sum_{\chi \in I_{\Gamma}}{d_{\chi}\delta^{\Gamma}_{\chi}} \in \mathbb{Z}_{\geq 0}$. Finally, a new statistic on $\Delta_{\Gamma}$ has been defined \cite[Definition $4.8$]{Pae1}. The group $W(\tilde{T}_{\Gamma})$ naturally acts by reflections on $\mathfrak{h}^*_{\Gamma}$. This action will be denoted by $*$. Define a new action of $W(\tilde{T}_{\Gamma})$ on $\Delta_{\Gamma}$ denoted by . such that
\begin{equation*}
 \omega*(\Lambda_{\chi_0}-d) = \Lambda_{\chi_0} - \omega.d, \qquad \forall (\omega,d) \in W(\tilde{T}_{\Gamma}) \times \Delta_{\Gamma}
\end{equation*}
where $\Lambda_{\chi_0}$ is the fundamental weight associated with $\chi_0$, the trivial character of $\Gamma$.\\
One can then prove that for each $d \in \Delta_{\Gamma}$, there exists a unique integer $r$ such that $d$ and $r\delta^{\Gamma}$ are in the $W(\tilde{T}_{\Gamma})$-orbit for the $.$ action \cite[Lemma 4.7]{Pae1}. Let us denote by $\mathrm{wt}(d)$ this integer $r$.\\
Recall the result of \cite[Theorem 4.10]{Pae1}, which will be our starting point. For each finite subgroup $\Gamma$ of $\SL_2(\C)$ we have indexed the irreducible components of $\Hk_n^{\Gamma}$ with the following set
\begin{equation*}
 \mathcal{A}^n_{\Gamma}:=\left\{d \in \Delta_{\Gamma}^+ \big{|} |d|_{\Gamma}=n \text{ and } \mathrm{wt}(d) \geq 0 \right\}.
\end{equation*}
Before diving into the type $D$ study, let us introduce a bit more notation. A partition $\lb$ of $n$ is a tuple  $(\lb_1 \geq \lb_2 \geq ... \geq \lb_r \geq 0)$ of integers, such that $|\lb|:=\sum_{i=1}^r{\lb_i}$ is equal to $n$. Denote by $\mathcal{P}_n$ the set of all partitions of $n$ and by $\mathcal{P}$ the set of all partitions of integers. For ${\lb=(\lb_1,...,\lb_r) \in \mathcal{P}}$, denote by  $\mathcal{Y}(\lb):=\{(i,j) \in \mathbb{Z}_{\geq 0}^2| i < \lb_1, j < r\}$ its associated Young diagram. The conjugate partition of a partition $\lb$ of $n$, denoted by $\lb^*$, is the partition associated with the reflection of $\mathcal{Y}(\lb)$ along the diagonal (which is again a Young diagram of a partition of $n$). We will draw Young diagrams upright and the box that is lowest and furthest to the left will have index $(0,0)$. Let $i$ and $j$ respectively  denote the row and column indices. For example, consider $\lb=(2,2,1)$. Its associated Young diagram is as follows
\begin{center}
\yng(1,2,2).
\end{center}
In that case  $\lb^*=(3,2)$.
A partition $\lb$ will be called symmetric if it is equal to its conjugate. Let us denote by $\mathcal{P}^s$ the set of all symmetric partitions and by  $\mathcal{P}^s_n:=\mathcal{P}^s \cap \mathcal{P}_n$.\\
A hook of a partition $\lb$ in position $(i,j) \in \mathcal{Y}(\lb)$ denoted by $H_{(i,j)}(\lb)$ is
\begin{equation*}
\left\{(a,b) \in \mathcal{Y}(\lb) \big| (a=i \text{ and } b \geq j) \text{ or }( a > i \text{ and } b=j)\right\}.
\end{equation*}
Define the length of a hook $H_{(i,j)}(\lb)$ to be its cardinal.
\begin{deff}
For a given integer $r \geq 1$, a partition $\lb$ is said to be an $r$-core  if $\mathcal{Y}(\lb)$ does not contain any hook of length $r$. Let us denote by $\mathfrak{C}_r$ the set of all $r$-cores and by  $\mathfrak{C}_r^s:= \mathfrak{C}_r \cap \mathcal{P}^s$.
\end{deff}

\section{From type $D$ to type $C$}
Fix $\ell \geq 2$, let $\mu_{\ell}$ denote the cyclic subgroup of $\SL_2(\C)$ generated by the diagonal matrix $\mathrm{diag}(\zeta_{\ell},\zeta_{\ell}^{-1})$, where $\zeta_{\ell}=e^{\frac{2i\pi}{\ell}}$. We will work with the following model of the binary dihedral group in $\SL_2(\C)$. Let $\BD_{2\ell}:=<\omega_{2\ell},s>$ where
\begin{multicols}{2}
\begin{itemize}
\item[] $\omega_{2\ell} :=
\begin{pmatrix}
\zeta_{2\ell} & 0\\
0 & \zeta^{-1}_{2\ell}
\end{pmatrix}$,
\item[] $s :=
\begin{pmatrix}
0 & \smin 1\\
1 & 0
\end{pmatrix}$.
\end{itemize}
\end{multicols}

\noindent The group $\BD_{2\ell}$ is of order $4\ell$. Note that $\BD_{4}$ is isomorphic to the quaternion group \cite[§ 1.7]{CM13}.
 Let $\tau_{2\ell}$ be the character of $\mu_{2\ell}$ that maps $\omega_{2\ell}$ to $\zeta_{2\ell}$. For $i \in \mathbb{Z}$, let
\begin{equation*}
\chi_i:=\mathrm{Ind}_{\mu_{2\ell}}^{\BD_{2\ell}}{\left(\tau_{2\ell}^i\right)}.
\end{equation*}
Note that $\chi_i$ is irreducible if and only if $i$ is not congruent to $0$ or $\ell$ modulo $2\ell$. If $\ell$ is even, the character table of $\BD_{2\ell}$ is

\begin{center}
\begin{tabular}{|c |c |>{\centering\arraybackslash}c|>{\centering\arraybackslash}c |>{\centering\arraybackslash}c|>{\centering\arraybackslash}c|}
\hline
cardinality & $1$ & $1$ & $2$ & $\ell$ & $\ell$ \\
\hline
classes & $\begin{pmatrix} 1 & 0 \\ 0 & 1 \end{pmatrix}$ & $\begin{pmatrix} \smin 1 & 0 \\ 0 & \smin 1 \end{pmatrix}$ & ${\omega_{2\ell}}^p (0 < p < \ell)$ & $s$ & $s\omega_{2\ell}$ \\
\hhline{|=|=|=|=|=|=|}
$\chi_{0^+}$ & $1$ & $1$ & $1$ & $1$ & $1$\\
\hline
$\chi_{0^{\smin}}$ & $1$ & $1$ & $1$ & $\smin 1$ & $\smin 1$\\
\hline
$\chi_{\ell^+}$ & $1$ & $1$ & $(\smin 1)^p$ & $\smin 1$ & $1$\\
\hline
$\chi_{\ell^{\smin}}$ & $1$ & $1$ & $(\smin 1)^p$ & $1$ & $\smin 1$\\
\hline
\thead{$\chi_k$ \\ $( 0<k<\ell)$}  & $2$ & $(\smin 1)^k2$ & $2 \mathrm{cos}\left(\frac{kp\pi}{\ell}\right)$ & $0$ & $0$\\
\hline
\end{tabular}
\end{center}

and if $\ell$ is odd, the character table of $\BD_{2\ell}$ is

\begin{center}
\begin{tabular}{|c |c |>{\centering\arraybackslash}c|>{\centering\arraybackslash}c |>{\centering\arraybackslash}c|>{\centering\arraybackslash}c|}
\hline
cardinality & $1$ & $1$ & $2$ & $\ell$ & $\ell$ \\
\hline
classes & $\begin{pmatrix} 1 & 0 \\ 0 & 1 \end{pmatrix}$ & $\begin{pmatrix} \smin 1 & 0 \\ 0 & \smin 1 \end{pmatrix}$ & ${\omega_{2\ell}}^p (0 < p < \ell)$ & $s$ & $s\omega_{2\ell}$ \\
\hhline{|=|=|=|=|=|=|}
$\chi_{0^+}$ & $1$ & $1$ & $1$ & $1$ & $1$\\
\hline
$\chi_{0^{\smin}}$ & $1$ & $1$ & $1$ & $\smin 1$ & $\smin 1$\\
\hline
$\chi_{\ell^+}$ & $1$ & $\smin1$ & $(\smin1)^p$ & $\zeta_4$ & $\smin\zeta_4$\\
\hline
$\chi_{\ell^{\smin}}$ & $1$ & $\smin1$ & $(\smin1)^p$ & $\smin\zeta_4$ & $\zeta_4$\\
\hline
\thead{$\chi_k$ \\ $( 0<k<\ell)$}  & $2$ & $(\smin1)^k2$ & $2 \mathrm{cos}\left(\frac{kp\pi}{\ell}\right)$ & $0$ & $0$\\
\hline
\end{tabular}.
\end{center}

\noindent The McKay graph of $\BD_{2\ell}$ is a Dynkin diagram of affine type $\Dt_{\ell+2}$
\begin{figure}[H]
\centering
\includegraphics[scale=0.3]{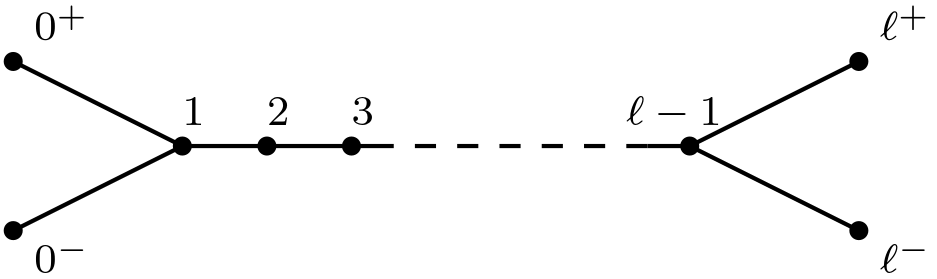}.
\end{figure}
\noindent The irreducible characters of $\BD_{2\ell}$ are labeled by their index in the McKay graph.\\
We want to give a combinatorial description of $\mathcal{A}^n_{\BD_{2\ell}}$. Let $\T_1$ denote the maximal diagonal torus of $\SL_2(\C)$. In what follows, we will give a combinatorial description using symmetric partitions of the irreducible components of $\Hk_n^{\BD_{2\ell}}$ containing a monomial ideal. To do so, restrict $\mathcal{A}^n_{\BD_{2\ell}}$ to the irreducible components of $\Hk_n^{\BD_{2\ell}}$ containing a $\T_1$-fixed point. Let us denote this subset of $\mathcal{A}^n_{\BD_{2\ell}}$ by $\ABDz$.
Note also that in this context, the coefficients of the null root in the base of simple roots are
\begin{equation*}
\deltD_{\chi_i}:= \begin{cases}
1 & \text{ if } i=0^+,\text{ }0^{\smin}, \text{ }\ell^+,\text{ }\ell^{\smin}\\
2 & \text{ otherwise}
\end{cases}.
\end{equation*}

\noindent The central object of study will be the affine root lattice of type $\Dt_{\ell+2}$ (which is the same object as the coroot lattice of type $\Dt_{\ell+2}$ since it is a simply laced type) $Q(\Dt_{\ell+2})\subset \mathfrak{h}^*_{\BD_{2\ell}}$. Let $\tau_{\ell}:=\alpha_{\chi_{0^{\smin}}}+\alpha_{\chi_{\ell^+}}+\alpha_{\chi_{\ell^{\smin}}}+\sum_{i=1}^{\ell\smin1}{2\alpha_{\chi_i}}$ be  the highest root of the finite root system of type $D_{\ell+2}$.

\begin{deff}
Define a bijection from the set $I_{\BD_{2\ell}}$ to itself
\begin{center}
$\sigma_{0^{\smin}} : \begin{array}{ccc}
I_{\BD_{2\ell}} & \to & I_{\BD_{2\ell}}\\
\chi & \mapsto & \chi_{0^{\smin}}.\chi  \\
\end{array}$
\end{center}
and define also an automorphism of the Dynkin diagram of type $\Dt_{\ell+2}$
\begin{center}
$\sigma : \begin{array}{ccc}
\Pi_{\BD_{2\ell}} & \to & \Pi_{\BD_{2\ell}}\\
\alpha_{\chi}& \mapsto & \alpha_{\sigma_{0^{\smin}}(\chi)}  \\
\end{array}$.
\end{center}
This automorphism swaps the first two vertices (the one with the label $0^+$ and $0^{\smin}$) and the last two (with the label $\ell^+$ and $\ell^{\smin}$) and fixes all the others.
\end{deff}

\noindent We can apply  Stembridge's  construction \cite{Stem08} to the root system of type $\Dt_{\ell+2}$ and to the automorphism $\sigma$.  Denote the simple roots $(\beta_i)_{i \in \llbracket 0,\ell\rrbracket}$ and $(\beta^{\vee}_i)_{i\in \llbracket 0,\ell\rrbracket}$ the simple coroots associated with the root system $\Phi(\Dt^{\sigma}_{\ell+2})$. By construction
\begin{multicols}{2}
\begin{itemize}
\item $\beta_0=\alpha_{\chi_{0^+}}+\alpha_{\chi_{0^{\smin}}}$
\item $\forall i \in \llbracket 1,\ell\smin1 \rrbracket$, $\beta_i=\alpha_{\chi_i}$
\item  $\beta_{\ell}= \alpha_{\chi_{\ell^+}}+\alpha_{\chi_{\ell^{\smin}}}$
\end{itemize}
\columnbreak
\begin{itemize}
\item $\beta^{\vee}_0=\frac{\alpha^{\vee}_{\chi_{0^+}}+\alpha^{\vee}_{\chi_{0^{\smin}}}}{2}$
\item  $\forall i \in \llbracket 1,\ell\smin1\rrbracket$, $\beta^{\vee}_i=\alpha^{\vee}_{\chi_i}$
\item $\beta^{\vee}_{\ell}=\frac{\alpha^{\vee}_{\chi_{\ell^+}}+\alpha^{\vee}_{\chi_{\ell^{\smin}}}}{2}$.
\end{itemize}
\end{multicols}

\noindent If $A=(a_{ij})$ is a generalized Cartan matrix, recall that in the associated Dynkin diagram, if two vertices $(i,j)$ are connected by more than one edge, then these edges are equipped with an arrow pointing toward $i$ if $|a_{ij}|>1$. With those conventions, the root system $\Phi(\Dt^{\sigma}_{\ell+2})$ has the following Dynkin diagram

\begin{figure}[h]
\centering
\includegraphics[scale=0.35]{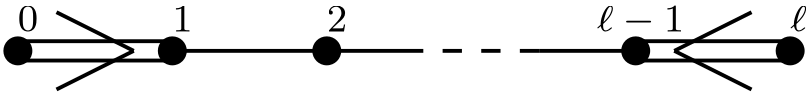}.
\end{figure}

\begin{prop}
The set $\Phi(\Dt_{\ell+2}^{\sigma})$ is a crystallographic root system of type $\Ct_{\ell}$.
\end{prop}

\begin{deff}
Let $Q(\Dt^{\sigma}_{\ell+2})[0^+]:= Q(\Dt^{\sigma}_{\ell+2}) \coprod \big(Q(\Dt^{\sigma}_{\ell+2}) + \alpha_{\chi_{0^+}}\big) \subset Q(\Dt_{\ell+2})$. Written more explicitly
\begin{equation*}
\QDts= \Set{\sum_{\chi\in I_{\BD_{2\ell}}}{a_{\chi}\alpha_{\chi}} \in Q(\Dt_{\ell+2}) |  0 \leq a_{\chi_0^+}-a_{\chi_0^{\smin}} \leq 1 \text{ and } a_{\chi_{\ell^+}}=a_{\chi_{\ell^{\smin}}}}
\end{equation*}

\end{deff}
\begin{deff}
Define the following map
\begin{center}
$\mathcal{T} :
\QDts \to  \Qc(\Dt_{\ell+2}^{\sigma})=\Qc(\Ct_{\ell})$
\end{center}
given by
\begin{center}
$\sum_{i=0}^{\ell}{a_i\beta_i} + q \alpha_{\chi_{0^+}}  \mapsto  (2a_0+q)\beta^{\vee}_0 + \sum_{i=1}^{\ell\smin1}{a_i\beta_i^{\vee}} + 2a_{\ell}\beta^{\vee}_{\ell}$
\end{center}

with $q \in \{0,1\}$.
\end{deff}

\noindent In type $\tilde{C}_{\ell}$, the null root is  ${\delta(\Ct_{\ell}):=\beta_0+\sum_{i=1}^{\ell\smin1}{2\beta_i}+\beta_{\ell} \in Q(\Ct_{\ell})}$ and the null coroot is  ${\delta^{\vee}(\Ct_{\ell}):=\sum_{i=0}^{\ell}{\beta^{\vee}_i}}$.  For each $\chi \in I_{\BD_{2\ell}}$, let $s_{\chi} \in W(\Dt_{\ell+2})$ denote the simple reflection associated with $\alpha_{\chi}$.
\begin{deff}
For each $\chi \in I_{\BD_{2\ell}}$, let $\sigma.s_{\chi}:=s_{\sigma_{0^{\smin}}(\chi)}$ and extend this action to $W(\Dt_{\ell+2})$, the Weyl group of type $\Dt_{\ell+2}$. Let $W(\Dt_{\ell+2})^{\sigma}:=\{\omega \in  W(\Dt_{\ell+2}) | \sigma.\omega=\omega \}$, which is a subgroup of $ W(\Dt_{\ell+2})$.
\end{deff}

\begin{rmq}
The set $\{s_0:=s_{\chi_{0^+}}s_{\chi_{0^{\smin}}},s_1:=s_{\chi_1},\dots,s_{\ell\smin1}:=s_{\chi_{\ell\smin1}},s_{\ell}:=s_{\chi_{\ell^+}}s_{\chi_{\ell^{\smin}}}\}$ is a set of generators of $W(\Dt_{\ell+2})^{\sigma}$. Applying \cite[Claim $3$]{Stem08} to our situation, gives a group isomorphism from  $W(\Dt_{\ell+2}^{\sigma})$ to $W(\Dt_{\ell+2})^{\sigma}$. Let us, from now on, identify these two groups and refer to them as $W(\Ct_{\ell})$. This group acts naturally by reflections on $\QDts$ and $\Qc(\Ct_{\ell})$. Denote this action by $*$.
\end{rmq}

\begin{deff}
Define a $W(\Ct_{\ell})$-action on $\QDts$ in the following way
\begin{equation*}
 s_i\blackdiamond \alpha := s_i*\alpha+\delta_i^{0}\alpha_{\chi_{0^+}}, \qquad \forall i \in \llbracket 0,\ell\rrbracket, \forall \alpha \in \QDts.
\end{equation*}
Define a $W(\Ct_{\ell})$-action on the coroot lattice $\Qc(\Dt^{\sigma}_{\ell+2})$ similarly
\begin{equation*}
 s_i\blackdiamond \beta^{\vee} := s_i*\beta^{\vee}+\delta_i^{0}\beta_0^{\vee}, \qquad \forall i \in \llbracket 0,\ell\rrbracket, \forall \beta^{\vee} \in \Qc(\Ct_{\ell}).
\end{equation*}
\end{deff}

\noindent
A simple computation shows the equivariance of $\mathcal{T}$ with respect to the former defined actions.
\begin{prop}
\label{T equi}
The map $\mathcal{T}$ is $W(\tilde{C}_{\ell})$-equivariant.
\end{prop}

\begin{rmq}
\label{T size}
Note also that $\mathcal{T}$ preserves sizes
\begin{equation*}
 |\alpha|_{\Dt_{\ell+2}} = |\mathcal{T}(\alpha)|_{\Ct_{\ell}},\quad \forall \alpha \in \QDts.
\end{equation*}
\end{rmq}

\noindent Let $G$ be an abstract group acting on a set $X$. For any $x \in X$, we denote by $\overline{x}^G$ the orbit of $x$ under the action of $G$. The following Lemma will be used later on when proving the first theorem.
\begin{lemme}
\label{lin cas part}
If $\beta^{\vee} \in \overline{0}^{W(\Ct_{\ell})}\subset Q^{\vee}(\tilde{C}_{\ell})$, and $k \in \mathbb{Z}$, then $(\beta^{\vee}+k\delta^{\vee}(\Ct_{\ell})) \in \overline{k\delta^{\vee}(\Ct_{\ell})}^{W(\Ct_{\ell})}$.
\end{lemme}
\begin{proof}
It is enough to check this on the set of generators $\{s_i | i \in \llbracket 0, \ell \rrbracket\}$. When $i \in \llbracket 1,\ell \rrbracket$, the action is by reflections. It is then linear and $s_i$ stabilizes $k\delta^{\vee}(\Ct_{\ell})$. For $i=0$, we can combine this fact
\begin{equation*}
s_0\blackdiamond (\beta^{\vee}_1 + \beta^{\vee}_2)= s_0 \blackdiamond \beta^{\vee}_1 + s_0 \blackdiamond \beta^{\vee}_2 -\beta^{\vee}_0,\qquad \forall \beta^{\vee}_1,\beta^{\vee}_2 \in \Qc(\Ct_{\ell})
\end{equation*}
with the fact that $s_0\blackdiamond \delta^{\vee}(\Ct_{\ell})=\delta^{\vee}(\Ct_{\ell})+\beta^{\vee}_0$ to conclude that
\begin{equation*}
s_0\blackdiamond(\beta^{\vee}+ k \delta^{\vee}(\Ct_{\ell}))=s_0\blackdiamond\beta^{\vee}+k\delta^{\vee}(\Ct_{\ell}).
\end{equation*}
\end{proof}

\noindent Finally, let us say a few words about the dual root system of $\Phi(\tilde{D}^{\sigma}_{\ell+2})$. It can be obtained as a folding of type $A$. This will simplify proofs in the next section. Recall that $\mu_{2\ell}$ denotes the cyclic subgroup of order $2\ell$ contained in the maximal diagonal torus of $\SL_2(\C)$ and that $\tau_{2\ell}$ denotes the irreducible character of $\mu_{2\ell}$ mapping the generator $\omega_{2\ell}$ to $\zeta_{2\ell}$. The McKay graph of $\mu_{2\ell}$ is a Dynkin diagram of affine type $\tilde{A}_{2\ell}$ with $2\ell$ vertices (since $\mu_{2\ell}$ is abelian). Consider the automorphism of the Dynkin diagram of type $\tilde{A}_{2\ell}$
\begin{center}
$\varsigma:\begin{array}{ccc}
\Pi_{\mu_{2\ell}} & \to & \Pi_{\mu_{2\ell}}\\
\alpha_{\tau^i} & \mapsto & \alpha_{\tau^{\smin i}}
\end{array}.$
\end{center}
It fixes $\alpha_{\tau^0}$ and $\alpha_{\tau^{\ell}}$. Applying Stembridge's construction to $(\tilde{A}_{2\ell},\varsigma)$ and identify it with the dual root system (cf. \cite[§3.1]{kac90}) of $\Phi(\tilde{D}_{\ell+2}^{\sigma})$.

\begin{prop}
\label{prop_dual}
The set $\Phi(\tilde{A}^{\varsigma}_{2\ell})$ is the dual root system of $\Phi(\tilde{D}_{\ell+2}^{\sigma})$.
\end{prop}

\section[Type $D$ Residue]{$\BD_{2\ell}$-Residue}
The $\T_1$-fixed points in $\Hk_n$ are the ideals $I_{\lb}$ generated by $\{x^iy^j | (i,j)\in \mathbb{N}^2\setminus{\mathcal{Y}(\lb)}\}$ for $\lb$ a partition of $n$. These ideals are called monomial ideals. Among these ideals, the ideals fixed by $s \in \BD_{2\ell}$ are exactly the monomial ideals parametrized by symmetric partitions of $n$. This implies that $\C[x,y]/I_{\lb}$ is a $\BD_{2\ell}$-module whenever $\lb$ is symmetric. In this section, our goal is to generalize the residue "of type $A$" i.e. the usual residue of partitions to a residue of type $D$. Recall that we identify the root lattice constructed out of $\Gamma$ with the Grothendieck ring of $\Gamma$. The property from the residue that we want to generalize is that the residue of a partition $\lb$ is equal to the character of the representation $\C[x,y]/I_{\lb}$. Thus, we want to construct a map $\mathrm{Res}_D$ from $\mathcal{P}^s_n$ to $Q(\Dt_{\ell+2})$.
To do so, let us first define the functions ${d_k: \mathcal{P}_n^s \to \mathbb{Z}_{\geq 0}}$, for each $k \in \llbracket 0, \ell \rrbracket$.\\
Let $\mathcal{Y}(\lb)_{k}:=\{(i,j) \in \mathcal{Y}(\lb)| i\smin j \equiv k \text{ }[2\ell]\}$ for $k \in \llbracket 0,2\ell\smin1\rrbracket$.
\begin{deff}
For $k \in \llbracket 1, \ell\rrbracket$ define
$d_k(\lb):= \#\big(\mathcal{Y}(\lb)_k \cup \mathcal{Y}(\lb)_{2\ell-k}\big)$. When $k=0$, consider $\tilde{d}_0(\lb):= \#\{(i,j) \in \mathcal{Y}(\lb)| i=j\}$  and $d_0(\lb) := \#\mathcal{Y}(\lb)_0- \tilde{d}_0(\lb)$.
\end{deff}

\noindent Denote by
\begin{multicols}{2}
\begin{itemize}
\item $d'_0(\lb):=\frac{d_0(\lb)}{2}+\tilde{d}_0(\lb)-\lfloor\frac{\tilde{d}_0(\lb)}{2}\rfloor$
\item $d''_0(\lb):=\frac{d_0(\lb)}{2} + \lfloor\frac{\tilde{d}_0(\lb)}{2}\rfloor$.
\end{itemize}
\end{multicols}
\noindent We are now able to define the residue in type $D$.
\begin{deff}
Let the residue of type $D$ be
\begin{center}
$\mathrm{Res}_D: \begin{array}{ccc}
\mathcal{P}^s_n & \to & Q(\Dt_{\ell+2})\\
\lb & \mapsto & d'_0(\lb) \alpha_{\chi_{0^+}} + d''_0(\lb) \alpha_{\chi_{0^{\smin}}} + \sum_{i=1}^{\ell\smin1}{\frac{d_i(\lb)}{2} \alpha_{\chi_i}} + \frac{d_{\ell}(\lb)}{2} (\alpha_{\chi_{\ell^+}}+\alpha_{\chi_{\ell^{\smin}}}).
\end{array}$
\end{center}
\end{deff}

\begin{rmq}
Using the fact that the partition is symmetric, it is easy to see that the image of $\mathrm{Res}_D$ is indeed in the $\mathbb{Z}$-span of the $\{\alpha_{\chi} | \chi\in I_{\BD_{2\ell}}\}$. Note moreover, that
\begin{equation*}
\forall \lb \in \mathcal{P}^s_n, |\mathrm{Res}_D(\lb)|_{\Dt_{\ell+2}}=|\lb|=n
\end{equation*}
\end{rmq}

\begin{ex}
Take $\ell=2$ and consider $\lb=(4,4,3,2)$ which is symmetric and has the following Young diagram

\begin{center}
\ytableaushort
{12, 21{0^+}, 1 {0^{\smin}}12, {0^+}121}
\end{center}

\noindent which gives that $\mathrm{Res}_D(\lb)=2\alpha_{\chi_{0^+}}+\alpha_{\chi_{0^{\smin}}}+3\alpha_{\chi_{1}}+2\alpha_{\chi_{2^+}}+2\alpha_{\chi_{2^{\smin}}}$.
\end{ex}

\begin{prop}
\label{gen res}
For any $\lb \in \mathcal{P}_n^s$, $\mathrm{Res}_D(\lb)$ is the character of the $\BD_{2\ell}$-representation $\C[x,y]/I_{\lb}$.
\end{prop}
\begin{proof}
Consider $(\overline{x^iy^j})_{(i,j) \in \mathcal{Y}(\lb)}$ a base of the representation  $\C[x,y]/I_{\lb}$. Since $\lb$ is symmetric, restrict the attention to $\mathcal{Y}^{-}(\lb):=\{(i,j) \in \mathcal{Y}(\lb) | i > j\}$ and to the diagonal ${\{(i,j) \in \mathcal{Y}(\lb)| i=j\}}$. Take first $(i,j) \in \mathcal{Y}^-(\lb)$ and consider ${V_{(i,j)}=\mathrm{Vect}(\overline{x^iy^j}, \overline{x^jy^i})}$ a subspace of $\C[x,y]/I_{\lb}$.
Let $k$ be an element of  $\llbracket 1, \ell\smin1 \rrbracket$. For each $(i,j) \in \mathcal{Y}^-(\lb)$ such that $i\smin j\equiv k [2\ell]$, we have $V_{(i,j)} \simeq_{\BD_{2\ell}} X_{\chi_k}$ (recall that $X_{\chi_k}$ is an irreducible representation  of $\BD_{2\ell}$ with character equal to $\chi_k$). Moreover when $i\smin j \equiv 2\ell\smin k [2\ell]$, we have $V_{(i,j)} \simeq_{\BD_{2\ell}} X_{\chi_k}$.
If $k=\ell$, then for each pair $(i,j) \in \mathcal{Y}^-(\lb)$ such that $i\smin j \equiv \ell [2\ell]$, we have $V_{(i,j)} \simeq_{\BD_{2\ell}} X_{\chi_{\ell^+}} \oplus X_{\chi_{\ell^{\smin}}}$. In the same way if $(i,j) \in \mathcal{Y}^-(\lb)$ such that $i\equiv j [2\ell]$, $V_{(i,j)} \simeq_{\BD_{2\ell}} X_{\chi_{0^+}} \oplus X_{\chi_{0^{\smin}}}$. It remains to understand the action of $\BD_{2\ell}$ on the diagonal. For each $i \in \mathbb{Z}_{\geq 0}, \omega_{2\ell}.\overline{x^iy^i}=\overline{x^iy^i}$ and $s.\overline{x^iy^i}=(\smin1)^i\text{ }\overline{x^iy^i}$. These two computations show that if $i \equiv 0 [2]$, then  $V_i:=V_{(i,i)} \simeq_{\BD_{2\ell}} X_{\chi_{0^+}}$ and that if $i \equiv 1 [2]$, then $V_i \simeq_{\BD_{2\ell}} X_{\chi_{0^{\smin}}}$. To sum it all up, the character of $\C[x,y]/I_{\lb}$ is $\mathrm{Res}_D(\lb)$.
\end{proof}

\noindent By construction $\mathrm{Res}_D$ factors though $\QDts$. For $a,b \in \mathbb{Z}$, let $\mathrm{rem}(a,b)\in \llbracket 0, b\smin 1 \rrbracket$ denote the remainder of the Euclidian division of $a$ by $b$.
Thanks to the work of Christopher R.H. Hanusa and Brant C. Jones \cite[Theoreom $5.8$]{Han12} we can endow the set $\mathfrak{C}_{2\ell}^s$ of symmetric $2\ell$-cores with a $W(\Ct_{\ell})$-action. Let us quickly recall how this action is constructed.
\begin{deff}
For a  symmetric $2\ell$-core $\lb$ define the $C$-residue of a box positioned at row $i$ and column $j$ in the Young diagram of $\lb$ as
\begin{center}
$\begin{cases}
\mathrm{rem}(j\smin i,2\ell) & \text{if } 0 \leq \mathrm{rem}(j\smin i,2\ell) \leq \ell\\
2\ell - \mathrm{rem}(j\smin i,2\ell) & \text{if } l < \mathrm{rem}(j\smin i,2\ell) < 2\ell.
\end{cases}$
\end{center}
\end{deff}

\begin{ex}
Take $\ell=2$ and the same symmetric $4$-core $(4,4,3,2)$. The Young diagram filled with the $C$-residue of each box gives
\begin{center}
\ytableaushort
{12, 210, 1012, 0121}
\end{center}
\end{ex}
\begin{rmq}
Note that for each symmetric $2\ell$-core $\lb$, the $C$-residue of each box of $\lb$ is always an integer between $0$ and $\ell$.
\end{rmq}

\begin{deff}
\label{action_core}
The action of $W(\Ct_{\ell})$ on $\mathfrak{C}_{2\ell}^s$ is defined on generators. Take ${s_i \in W(\Ct_{\ell})}$ and $\lb \in \mathfrak{C}_{2\ell}^s$. Note that there are only three disjoint cases. Either we can add boxes with $C$-residue $i$, or we can remove such boxes or there are no such boxes. Define $s_i.\lb$ as the partition obtained from $\lb$ in either adding all boxes of $\mathcal{Y}(\lb)$ with $C$-residue $i$ so that $s_i.\lb$ remains a partition
or removing all boxes of $\mathcal{Y}(\lb)$ with $C$-residue $i$ so that $s_i.\lb$ remains a partition.
\end{deff}

\begin{deff}
The $C$-region of index $k \in \mathbb{Z}$ of a symmetric $2\ell$-core is the following subset of $\mathcal{Y}(\lb)$
\begin{equation*}
\mathcal{R}_k:=\{(i,j) \in \mathcal{Y}(\lb) | (i\smin j) \in \{2k\ell,...,2(k+1)\ell\smin 1\}\}
\end{equation*}
More generally, we can define a shifted $C$-region. Let $(k,h) \in \mathbb{Z}^2$ and define the $h$-shifted $C$-region of index $k$
\begin{equation*}
\mathcal{R}_{k,h}:=\{(i,j) \in \mathcal{Y}(\lb) | (i\smin j) \in \{2k\ell+h,...,2(k+1)l \smin 1+h\}\}
\end{equation*}
\end{deff}

\begin{prop}
\label{Res equi}
$\mathrm{Res}_D: \mathfrak{C}_{2\ell}^s \to \QDts$ is $W(\tilde{C}_{\ell})$-equivariant.
\end{prop}
\begin{proof}
Thanks to Proposition \ref{prop_dual}, we have $\QDts \subset Q^{\vee}(\tilde{A}_{2\ell})$. Moreover, the type $\tilde{A}_{2\ell}$ is simply laced. We can thus identify $Q^{\vee}(\tilde{A}_{2\ell})$ with $Q(\tilde{A}_{2\ell})$. Using Proposition \ref{prop_dual}, we can also identify $W(\tilde{A}^{\varsigma}_{2\ell})$ with $W(\tilde{C}_{\ell})$. Now, thanks to \cite[Proposition 3.2.5]{BJV09}, we have that usual residue map, $\Res:\mathfrak{C}_{2\ell} \to Q(\tilde{A}_{2\ell})$ is $W(\tilde{A}_{2\ell})$-equivariant. Finally, the restriction of this map to $\mathfrak{C}_{2\ell}^s$ gives $\mathrm{Res}_D:\mathfrak{C}_{2\ell}^s \to \QDts$. Indeed, using Proposition \ref{gen res} and the definition of the irreducible characters of $\BD_{2\ell}$ we see that it is already true that $\Res_D$ is the restriction of the usual residue to symmetric partitions. We can thus conclude that $\mathrm{Res}_D: \mathfrak{C}_{2\ell}^s \to \QDts$ is $W(\tilde{A}^{\varsigma}_{2\ell})$-equivariant.
\end{proof}

\begin{prop}
 \label{bij 0}
$\mathcal{T} \circ \mathrm{Res}_D : \mathfrak{C}_{2\ell}^s \to \overline{0}^{W(\tilde{C}_{\ell})} \subset \Qc(\tilde{C}_{\ell})$ is a bijection.
\end{prop}
\begin{proof}
By definition, we have $\mathcal{T}(\mathrm{Res}_D(\emptyset))=0$ and the stabilizer of $\emptyset \in \mathfrak{C}_{2\ell}^s$ in $W(\tilde{C}_{\ell})$ is equal to $W(C_{\ell})$, the Weyl group of the finite type $C_{\ell}$, which is equal to the stabilizer of $0 \in \overline{0}^{W(\tilde{C}_{\ell})}$ in  $W(\tilde{C}_{\ell})$. Moreover, using Proposition \ref{T equi} and Proposition \ref{Res equi}, we know that $\mathcal{T}\circ \mathrm{Res}_D$ is $W(\tilde{C}_{\ell})$-equivariant. To conclude, it is enough to show that the $W(\tilde{C}_{\ell})$-action defined on $\mathfrak{C}_{2\ell}^s$ (Definition \ref{action_core}) is transitive. This has been proven in \cite[Proposition $6.2$]{Han12}.
\end{proof}

\begin{rmq}
Note that the Proposition \ref{bij 0} can also be deduced from \cite[Proposition 4.4]{BM21} and Proposition \ref{prop_dual}.
\end{rmq}

\begin{prop}
 The following composition of maps
\begin{center}
\begin{tikzcd}
\upvarphi:  \mathfrak{C}_{2\ell}^s \ar[r,"\mathcal{T} \circ \mathrm{Res}_D"] & \Qc(\Ct_{\ell}) \ar[r,"\pi",two heads]& \Qc(\Ct_{\ell})/\mathbb{Z}\delta^{\vee}(\Ct_{\ell})
\end{tikzcd}
\end{center}
is a bijection.
\end{prop}
\begin{proof}
Consider the bijection $\overline{0}^{W(\Ct_{\ell})} \iso \Qc(C_{\ell})$ which is the composition of these two bijections
\begin{equation*}
\overline{0}^{W(\Ct_{\ell})} \iso W(\Ct_{\ell})/W(C_{\ell}) \iso \Qc(C_{\ell})
\end{equation*}
The second bijection boils down to the choice of a representative with coordinate $0$ along $\beta^{\vee}_0$. Moreover, consider the bijection
\begin{center}
$\begin{array}{ccc}
\Qc(\Ct_{\ell})/\mathbb{Z}\delta^{\vee}(\Ct_{\ell}) & \iso & \Qc(C_{\ell})  \\
\beta^{\vee}  & \mapsto & \beta^{\vee} - \beta^{\vee}_{0} \delta^{\vee}(\Ct_{\ell}) \\
\end{array}$.
\end{center}
We then have the following commutative diagram
\begin{center}
\begin{tikzcd}
\overline{0}^{W(\Ct_{\ell})}\ar[dr, "\sim" labld] \ar[rr,"\pi", two heads]  & & \Qc(\Ct_{\ell})/\mathbb{Z}\delta^{\vee}(\Ct_{\ell}) \ar[ld, "\sim"' lablad]\\
&\Qc(C_{\ell})&
\end{tikzcd}.
\end{center}
From there, we can use Proposition \ref{bij 0} to prove that $\upvarphi$ is  a bijection.
\end{proof}

\section{Combinatorial description in type $D$}
We now have everything needed to give a combinatorial description of the set $\ABDz$. Note that, thanks to Proposition \ref{gen res},  $\ABDz \subset \mathcal{A}^n_{\BD_{2\ell}} \cap \QDts$.\\
Consider the following map
\begin{center}
$\begin{array}{ccc}
\epsilon : \QDts & \to & \mathfrak{C}_{2\ell}^s\\
d & \mapsto & (\upvarphi^{-1}\circ \pi \circ\mathcal{T})(d) \\
\end{array}$.
\end{center}

\begin{thm}
\label{param_ABDz}
The map $\epsilon$ defines a bijection between $\ABDz$ and the symmetric $2\ell$-cores $\lb$, such that ${|\lb|\equiv n\text{ } [2\ell]}$ and $|\lb| \leq n$. Moreover, for each $\mu_1,\mu_2 \in \mathcal{P}^s_n$, $I_{\mu_1}$ and $I_{\mu_2}$ are in the same irreducible component of $\Hk_n^{\BD_{2\ell}}$ if and only if the $2\ell$-cores of $\mu_1$ and $\mu_2$ are equal.
\end{thm}
\begin{proof}
First let us show that if $d \in \ABDz$ then $|\epsilon(d)| \equiv n\text{ } [2\ell]$. Denote $\lb:=\epsilon(d)$, then
\begin{equation*}
\exists k \in \mathbb{Z},\mathcal{T}(d)=\mathcal{T}(\mathrm{Res}_D(\lb)) + k \delta^{\vee}(\Ct_{\ell}).
\end{equation*}
 In particular $|\mathcal{T}(d)|_{\Ct_{\ell}}=|\mathcal{T}(\mathrm{Res}_D(\lb))|_{\Ct_{\ell}} + 2k\ell$. Now since $d \in \ABDz$, $|d|_{\Dt_{\ell+2}}=n$ and using Remark \ref{T size} we have that $n=|\lb| + 2k\ell$. Moreover, let us show that if ${d \in \ABDz}$, then $|\epsilon(d)|\leq n$. Thanks to Lemma \ref{lin cas part} and to the fact that $\mathcal{T}\left(\mathrm{Res}_D(\lb)\right) \in \overline{0}^{W(\Ct_{\ell})}$,  we have that ${\mathcal{T}(d) \in \overline{k\delta^{\vee}(\Ct_{\ell})}^{W(\Ct_{\ell})}}$. Since $\mathrm{wt}(d) \geq 0$, there exists ${k' \in \mathbb{Z}_{\geq 0}}$ such that $d \in \overline{k'\delta(\Dt_{\ell+2})}^{W(\Dt_{\ell+2})}$. In fact $d \in \overline{k'\delta(\Dt_{\ell+2})}^{W(\Ct_{\ell})}$ since $d \in \QDts$. The map $\mathcal{T}$ sends $\delta(\Dt_{\ell+2})$ to $2\delta^{\vee}(\Ct_{\ell})$, which  then gives that  $\mathcal{T}(d) \in \overline{2k'\delta^{\vee}(\Ct_{\ell})}^{W(\Ct_{\ell})}$ and so $k=2k'$, by construction of $\mathrm{wt}(d)$ (cf. \cite[Lemma 4.7]{Pae1}). Since $n=|\lb|+2k\ell$, we have that $k \geq 0 \iff |\lb| \leq n$.
The map ${\epsilon:\ABDz \to \{\lambda \in \mathfrak{C}_{2\ell}^s \big| |\lb| \equiv n \text{ }[2\ell], |\lb|\leq n\}}$ has now been proven to be well defined. By construction, $\epsilon$ is the converse map of $\mathrm{Res}_D$ and establishes a bijection between $\ABDz$ and ${\{\lambda \in \mathfrak{C}_{2\ell}^s \big| |\lb| \equiv n \text{ }[2\ell], |\lb|\leq n\}}$.\\
Concerning the second assertion,  we have that $I_{\mu_1}$ and $I_{\mu_2}$ are in the same irreducible component of $\Hk_n^{\BD_{2\ell}}$ if and only if the character of $\C[x,y]/I_{\mu_1}$ is equal to the character of $\C[x,y]/I_{\mu_2}$, thanks to \cite[Corollary $4.3$]{Pae1}. Now using Proposition \ref{gen res}, we know that it is the case if and only if $\mathrm{Res}_D(\mu_1)=\mathrm{Res}_D(\mu_2)$. By construction, for each ${i \in \{1,2\}}$, $\epsilon(\mathrm{Res}_D(\mu_i))$ is the $2\ell$-core of $\mu_i$ which then gives the result.
\end{proof}

\begin{rmq}
Take $d \in \ABDz$ and $\lb\in \mathcal{P}^s_n$ such that $I_{\lb}$ is in the irreducible component parametrized by $d$. Let $\gamma_{2\ell}$ denote the $2\ell$-core of $\lb$. We have, as a by-product of the proof of Theorem \ref{param_ABDz}, that $\frac{n-|\gamma_{2\ell}|}{2\ell}$, which is the number of $2\ell$-hooks that we need to remove from $\lb$ to obtain its $2\ell$-core, is equal to $2\mathrm{wt}(d)$.
\end{rmq}

\begin{ex}
The set $\ABDz$ is a proper subset of $\mathcal{A}^n_{\BD_{2\ell}}$. If $\ell=2$, we can find for each $r \in \mathbb{Z}_{>0}$ an irreducible component of $\Hk_{8r+4}^{\BD_4}$ of dimension $2r$ that is parametrized by an element of $\mathcal{A}^{8r+4}_{\BD_{4}} \setminus \mathcal{A}^{8r+4, \T_1}_{\BD_{4}}$. Let $\omega=s_{\chi_{2^+}}s_{\chi_{1}}s_{\chi_{0^+}}\in W(\tilde{T}_{\BD_4})$ and consider $\omega.r\delta^{\BD_4}$. We have that ${(\omega.r\delta^{\BD_4})}_{\chi_{2^+}}={(\omega.r\delta^{\BD_4})}_{\chi_{2^-}}+1$, which implies that this element is not in $\mathcal{A}^{8r+4,\T_1}_{\BD_4}$ thanks to Proposition \ref{gen res}.
\end{ex}

\section{Absence of combinatorics in type $E$}
The binary tetrahedral group $\tilde{A}_{4}$ is a central extension of $A_4$, the alternating group on $4$ elements \cite[§ 6.5]{CM13}. It has order $24$ and has the following presentation
\begin{equation*}
<a,b,c\text{ }|\text{ }a^2=b^3=c^3=abc>.
\end{equation*}

\noindent Let us denote by $z:=abc$ which is a central element of $\tilde{A}_4$. Note that $z$ has order $2$. The group $\tilde{A}_4$ has the following character table.

\begin{center}
\begin{tabular}{|c |c |>{\centering\arraybackslash}c|>{\centering\arraybackslash}c |>{\centering\arraybackslash}c|>{\centering\arraybackslash}c|>{\centering\arraybackslash}c|>{\centering\arraybackslash}c|}
\hline
cardinality & $1$ & $1$ & $6$ & $4$ & $4$ & $4$& $4$ \\
\hline
classes & $\begin{pmatrix} 1 & 0 \\ 0 & 1 \end{pmatrix}$ & $z$ & $a$ & $b$ & $c$ & $b^2$ & $c^2$ \\
\hhline{|=|=|=|=|=|=|=|=|}
$\chi_0$ & $1$ & $1$ & $1$ & $1$ & $1$ & $1$ & $1$\\
\hline
$\psi$ & $1$ & $1$ & $1$ & $\zeta_3$ & $\zeta_3^2$ & $\zeta_3^2$ & $\zeta_3$\\
\hline
$\psi^2$ & $1$ & $1$ & $1$ & $\zeta_3^2$ & $\zeta_3$ & $\zeta_3$ & $\zeta_3^2$\\
\hline
$\mathrm{X}$ & $3$ & $3$ & $\smin 1$ & $0$ & $0$ & $0$ & $0$\\
\hline
$\chi_{\std}$& $2$ & $\smin 2$ & $0$ & $1$& $1$ & $\smin 1$ & $\smin 1$\\
\hline
$\psi\chi_{\std}$& $2$ & $\smin 2$ & $0$ & $\zeta_3$ & $\zeta_3^2$ & $\smin \zeta_3^2$ & $\smin \zeta_3$\\
\hline
$\psi^2\chi_{\std}$& $2$ & $\smin 2$ & $0$ & $\zeta_3^2$ & $\zeta_3$ & $\smin \zeta_3$ & $\smin \zeta_3^2$\\
\hline
\end{tabular}
\end{center}

\noindent The McKay graph of any finite subgroup of $\SL_2(\C)$ isomorphic to $\tilde{A}_4$ is  of affine type $\tilde{E}_6$
\begin{figure}[H]
\centering
\includegraphics[scale=0.45]{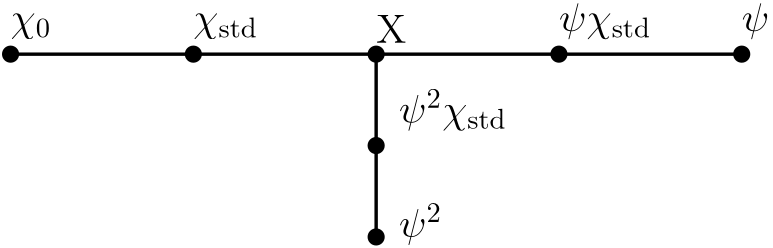}.
\end{figure}

\noindent The goal is here to study the combinatorics of the irreducible components of $\Hk_n^{\Gamma}$, when $\Gamma$ is of type $\tilde{E}_6$ (meaning that $\Gamma$ is isomorphic to $\tilde{A}_4$). Let us show that the irreducible components  containing a monomial ideal are fixed under $\SL_2(\C)$. Let $X_{\std}$ denote the standard representation of $\SL_2(\C)$ with its canonical basis $(e_1,e_2)$ and denote by $B_1$, respectively $B_2$ the stabilizer of $e_1$ respectively $e_2$ in $\SL_2(\C)$. The subgroups $B_1$ and $B_2$ are the two Borel subgroups of $\SL_2(\C)$ containing $\mathbb{T}_1$. Let us fix $\Gamma$ a finite subgroup of type $\tilde{E}_6$ in $\SL_2(\C)$.


\begin{lemme}
\label{no_norm}
The group $\Gamma$ is not conjugate to any subgroup of the normalizer of $\mathbb{T}_1$ in $\SL_2(\C)$ denoted by $N_{\SL_2(\C)}(\mathbb{T}_1)$. Furthermore, the group $\Gamma$ is neither conjugate to  a subgroup of $B_1$ nor of $B_2$.
\end{lemme}
\begin{proof}
The representation ${X_{\std}\otimes X_{\std}^*}$ is isomorphic to the direct sum of the trivial representation (generated by $e_1\otimes e_1^* + e_2 \otimes e_2^*$) and the adjoint representation of $\SL_2(\C)$. On the one hand, note that for the character $\chi_{\mathrm{std}}$ of $\Gamma$, we have that $\langle {(\chi_{\std})}^2,{(\chi_{\std})}^2\rangle =2$ which implies that the restriction of the adjoint representation to $\Gamma$ is irreducible. On the other hand, the restriction of the adjoint representation to $N_{\SL_2(\C)}(\mathbb{T}_1)$ is not irreducible since the one-dimensional subspace of $X_{\std} \otimes X_{\std}^*$ generated by $e_1\otimes e_1^* - e_2 \otimes e_2^*$ is $N_{\SL_2(\C)}(\mathbb{T}_1)$-stable. Moreover, the one-dimensional subspace of $X_{\std}$ generated by $e_1$  is $B_1$-stable and the one-dimensional subspace of $X_{\std}$ generated  by $e_2$ is $B_2$-stable.
\end{proof}

\begin{prop}
The subgroup $G$ of $\SL_2(\C)$ generated by $\mathbb{T}_1$ and $\Gamma$ is $\SL_2(\C)$.
\end{prop}
\begin{proof}
Thanks to Lemma \ref{no_norm}, there exists $x \in \Gamma$ such that $\mathbb{T}_1 \neq x\mathbb{T}_1x^{-1}$. We then have that the two subgroups $\mathbb{T}_1$ and $x\mathbb{T}_1x^{-1}$ are both irreducible and  connected subgroups of $\SL_2(\C)$. Let us denote by $H$ the subgroup of $\SL_2(\C)$ generated by these two one-dimensional tori.
Thanks to \cite[section $7.5$]{Hump75}, we know that $H$ is a closed connected subgroup of $\SL_2(\C)$. Since $H$ is not equal to $\SL_2(\C)$, and is of dimension at least two, $H$ is of dimension $2$. Using \cite[Corollary $11.6$]{Bor12} we know that $H$ is solvable. The algebraic group $H$ is then a Borel subgroup of $\SL_2(\C)$ containing $\mathbb{T}_1$ and contained in $G$. Moreover, thanks to the Bruhat decomposition \cite[Theorem $14.12$]{Bor12}, we know that $\SL_2(\C)=B_1sB_1 \coprod B_1=B_2sB_2 \coprod B_2$. Combining the Bruhat decomposition with Lemma \ref{no_norm}, this gives that $s \in G$.
Thanks to \cite[Proposition $11.19$]{Bor12}, we know that all Borel subgroups containing $\mathbb{T}_1$ are conjugated by the Weyl group of $\mathbb{T}_1$ denoted by $W(\mathbb{T}_1)$ which is by construction, the group generated by $\bar{s} \in W(\mathbb{T}_1)$. This implies that all Borel subgroups containing $\mathbb{T}_1$ are in $G$. Finally, using \cite[Proposition $13.7$]{Bor12}, we have that $G=\SL_2(\C)$.

\end{proof}

\begin{deff}
A partition is called a staircase partition if there exists a certain integer $m$ such that it is equal to $\lb_m:=(m,m\smin 1,...,1) \vdash \frac{m(m+1)}{2}$. Note that $\mathfrak{C}_2$ is equal to the set of all staircase partitions.
\end{deff}

\begin{prop}
The only $\SL_2(\C)$ fixed points of $\Hk_n$ are the monomial ideals associated with staircase partitions of size $n$.
\end{prop}
\begin{proof}
We already know that $\mathbb{T}_1$-fixed points are exactly monomial ideals. Moreover, thanks to \cite[Lemma $12$]{SJ02}, we have that the fixed points under the subgroup $\mathbb{B}_2$ of $\GL_2(\C)$ consisting of all upper triangular matrices are parametrized by staircase partitions. Let $\T_2$ be the maximal diagonal torus of $\GL_2(\C)$. Since $\mathbb{B}_2=\mathbb{T}_2B_1$, we get that $B_1$-fixed points of $\Hk_n$ are also parametrized by staircase partitions and the result follows.
\end{proof}

\noindent Finally, the binary octahedral group (type $\tilde{E}_7$) and the binary icosahedral group (type $\tilde{E}_8$) both contain a subgroup isomorphic to $\tilde{A}_{4}$ which then implies that the combinatorics of fixed points which are also $\T_1$-fixed is the same as the one of $\SL_2(\C)$. We then have proven the following result.

\begin{prop}
\label{prop SL2 fixed}
If $\Gamma$ is a finite subgroup of $\SL_2(\C)$ of type $\tilde{E}_6$, $\tilde{E}_7$ or $\tilde{E}_8$, then  for each ${n \in \mathbb{Z}_{\geq 1}}$, there is at most one irreducible component of $\Hk_n^{\Gamma}$ containing a $\T_1$-fixed point and it is indexed by the staircase partition of size $n$ (when it exists).
\end{prop}

\section{Dimension of the irreducible components containing a $\T_1$-fixed point}

In this section we will show that each irreducible component of $\Hk_n^{\Gamma}$ containing a $\T_1$-fixed point is zero-dimensional, whenever $\Gamma$ is of type $\tilde{E}_6$ in $\SL_2(\C)$. Thanks to the Proposition \ref{prop SL2 fixed}, we know that it enough to compute the dimensions of the irreducible components of $\Hk_n^{\Gamma}$ which contain a $\T_1$-fixed point indexed by a staircase partition. The results of this section will not depend on the choice of $\Gamma$ but only on the McKay graph. Since we need to make explicit computations, let us work with the following model of the binary tetrahedral group. Let $t \in \SL_2(\C)$ be the matrix
\begin{equation*}
\frac{1}{\sqrt{2}}\begin{pmatrix}
\zeta_8 & \zeta_8\\
\zeta_8^3 & \zeta_8^{-1}
\end{pmatrix}.
\end{equation*}
Consider the subgroup of $\SL_2(\C)$ generated by $\omega_4$, $s$ and $t$. Let us denote this group by $\BT$. By setting $a=s\omega_4$, $b=t$ and $c=st^2$, one can show that $\BT$ has the desired presentation (namely the one of $\tilde{A}_4$). Note also that $\BT=\BD_{4} \rtimes \langle t^2 \rangle$.

\subsection{$\tilde{E}_6$-Residue}

The irreducible components of $\Hk_n^{\Gamma}$ are isomorphic to quiver varieties over the doubled, framed McKay quiver \cite[Proposition 3.19]{Pae1}. Since we are interested in the irreducible components containing a $\T_1$-fixed point. Take $m \in \Z_{\geq 1}$, we know that the dimension parameter of this quiver variety is equal to the character of $\BT$ of the representation $\C[x,y]/I_{\lb_m}$. In this subsection, we will then construct a map $\mathrm{Res_{\tilde{E}_6}}: \mathfrak{C}_2 \to Q(\tilde{E}_6)$ which computes the decomposition into irreducible characters of the character of the representations $\C[x,y]/I_{\lb_m}$ for each $\lb_m \in \mathfrak{C}_2$. To do that, let us first give the notation.
 If $m=2k \in \Z_{\geq 0}$, define
\begin{equation*}
d^0_m:=
\begin{cases}
1+\floor{\frac{k\smin 2}{3}} &\text{if } m\equiv 0 [3]\\
\floor{\frac{k}{3}} &\text{if } m\equiv 1 [3]\\
1+ \floor{\frac{k\smin 1}{3}} &\text{if } m\equiv 2 [3]
\end{cases}.
\end{equation*}
\noindent Let $d_m:=\frac{m \smin  2d_m^0}{4}$. The fact that $d_m \in \mathbb{Z}_{\geq 0}$ results from the definition of $d^0_m$.
If now $m=2k+1 \in \Z_{\geq 0}$, define
\begin{equation*}
a_m:= 1+\floor{\frac{k\smin 1}{2}}.
\end{equation*}
Let $b_m:=m \smin 3a_m$. Moreover define
\begin{equation*}
e_m^0:=
\begin{cases}
\floor{\frac{b_m+1}{3}} &\text{if }  m \equiv 0 [3]\\
1+\floor{\frac{b_m}{3}} &\text{if }  m \equiv 1 [3]\\
\hspace*{0.2cm} \floor{\frac{b_m}{3}} &\text{if }  m \equiv 2 [3]
\end{cases}.
\end{equation*}

\noindent Let $e_m:=\frac{b_m\smin e_m^0}{2}$. The fact that $e_m \in \mathbb{Z}_{\geq 0}$ results from the definition of $e^0_m$. For the sake of clarity, let us introduce
\begin{equation*}
\beta_m:=
\begin{cases}
d_m^0\alpha_{\chi_{\std}}+d_m\alpha_{\psi\chi_{\std}}+ d_m \alpha_{\psi^2\chi_{\std}} & \text{ if } m \text{ is even}\\
a_m\alpha_{\Chi}+e_m^0\alpha_{\chi_0}+e_m\alpha_{\psi}+e_m \alpha_{\psi^2} & \text{ else}
\end{cases} \qquad, \forall m \in \Z_{\geq 0}.
\end{equation*}

\noindent We define $\mathrm{Res}_{\tilde{E}_6}$ such that the difference between $\mathrm{Res}_{\tilde{E}_6}(\lb_m)$ and $\mathrm{Res}_{\tilde{E}_6}(\lb_{m\smin 1})$ is exactly the element $\beta_m$ of the $\tilde{E}_6$-root lattice.

\begin{deff}
\label{res_e6}
Define the map $\tilde{E}_6$-Residue in the following way
\begin{equation*}
\mathrm{Res}_{\tilde{E}_6}:\begin{array}{ccc}
\mathfrak{C}_2 & \to & Q(\tilde{E}_6)\\
\lambda_0 & \mapsto & 0\\
\lambda_m & \mapsto & \mathrm{Res}_{\tilde{E}_6}(\lambda_{m \smin 1}) + \beta_m
\end{array}.
\end{equation*}
\end{deff}

\begin{ex}
Take $\lb_3=(3,2,1)$. Its Young diagram is filled as follows
\begin{center}
\ytableausetup{boxsize=2em}
\ytableaushort
{{\Chi}, {\chi_{\std}}{\Chi}, {\chi_{0}}{\chi_{\std}}{\Chi}}.
\end{center}
This gives $\Res_{\tilde{E}_6}(\lb_3)=\alpha_{\chi_0}+\alpha_{\chi}+\alpha_{\Chi}$.
\end{ex}

The following proposition assures us that $\Res_{\tilde{E}_6}$ fulfills its purpose.

\begin{prop}
\label{res_useful}
For each $\lb_m \in \mathfrak{C}_2$, $\mathrm{Res}_{\tilde{E}_6}(\lb_m)$ is the character of the $\BT$-representation $\C[x,y]/I_{\lb_m}$.
\end{prop}
\begin{proof}
To decompose the character of the $\SL_2(\C)$-representation $\C[x,y]/I_{\lb_m}$ along the irreducible characters of $\BT$, we will use the fact that $\BT=\BD_4 \rtimes \langle t^2 \rangle$. The group $\langle t^2 \rangle$ is conjugated to $\mu_3$ in $\SL_2(\C)$. Moreover, we deduce from the character tables and Clifford theory \cite[Theorem $6.2$]{I11} that $\Chi=\mathrm{Ind}_{\BD_4}^{\BT}\left(\chi_{2^+}\right)$. Thanks to Proposition \ref{gen res}, we deduce that the recursive Definition \ref{res_e6} is the character of the $\BT$-representation $\C[x,y]/I_{\lb_m}$.
\end{proof}

\noindent Now that we have computed the decomposition of the character of $\C[x,y]/I_{\lb_m}$ for each $\lb_m \in \mathfrak{C}_2$, we need to define the Euler form to compute the dimension of the irreducible component of $\Hk_{\frac{m(m+1)}{2}}^{\BT}$ containing $I_{\lb_m}$. To define this form, one needs to choose an orientation of the McKay quiver. Let us work with this orientation:
\begin{figure}[H]
\centering
\includegraphics[scale=0.45]{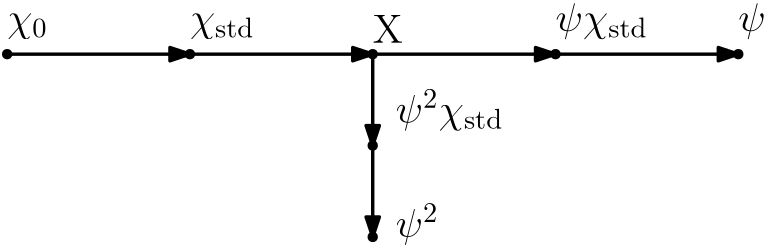}
\end{figure}

\noindent Let $E_{\tilde{E}_6}$ be the set of oriented arrows of the McKay quiver $\tilde{E}_6$. For an arrow $h \in E_{\tilde{E}_6}$, we will respectively denote by $h'$ and $h''$ the source and target of $h$.

\subsection{ Zero-dimensionnal irreducible components}
\begin{deff}
The Euler form is a bilinear form defined on the root lattice (which is identified with the lattice of dimension parameters) in the following way:
\begin{equation*}
\langle v,w\rangle:=\sum_{\chi \in I_{\tilde{E}_6}}{v_{\chi}w_{\chi}} - \sum_{h \in E_{\tilde{E}_6}}{v_{h'}w_{h''}}, \quad \forall (v,w) \in Q(\tilde{E}_6)^2.
\end{equation*}

\end{deff}

\begin{rmq}
Our results will only involve the Tits form (which is the associated quadratic form). Thus they will not depend on the choice of an orientation. Only the intermediate computations will use the Euler form.
\end{rmq}

\begin{thm}
\label{thm_dim_0}
For each $m \in \mathbb{Z}_{\geq 0}$, the irreducible component of $\Hk_{\frac{m(m+1)}{2}}^{\BT}$ containing $I_{\lb_m}$ is of dimension 0.
\end{thm}
\begin{proof}
Combining \cite[Proposition $3.19$]{Pae1} and Proposition \ref{res_useful} the irreducible component of $\Hk_{\frac{m(m+1)}{2}}^{\BT}$ containing $I_{\lb_m}$ is isomorphic to the quiver variety on the McKay quiver with dimension parameter $\Res_{\tilde{E}_6}(\lb_m)$.
Thanks to \cite[Corollary $3.12$]{Nak98}, the dimension of this quiver variety is equal to
\begin{equation*}
{2\left({\Res_{\tilde{E}_6}(\lb_m)}_{\chi_0}- \langle \Res_{\tilde{E}_6}(\lb_m),\Res_{\tilde{E}_6}(\lb_m) \rangle\right)}.
\end{equation*}
There remains to prove that this integer is equal to zero. To improve readability, we prove the remaining equality in Proposition \ref{prop_rel_res}.
\end{proof}

\noindent Before being able to finish the proof of Theorem \ref{thm_dim_0}, we need to prove a technical lemma.

\begin{lemme}
For each $k \in \mathbb{Z}_{\geq 1}$, we have
\begin{equation}
\label{lemme_r1}
d^0_{2(k+1)}+d^0_{2k}=e^0_{2k+1}+a_{2k+1}
\end{equation}

\begin{equation}
\label{lemme_r2}
e^0_{2k\smin 1}+e^0_{2k+1}=d^0_{2k}
\end{equation}

\begin{equation}
\label{lemme_r3}
k+1=a_{2k+1}+a_{2k+3}.
\end{equation}

\end{lemme}
\begin{proof}
To prove relation (\ref{lemme_r1}), let us consider the following cases.
\begin{itemize}
\item If ${2k\smin 1\equiv 0 [3]}$, then $d^0_{2(k+1)}+d^0_{2k}=1+\floor{\frac{k-1}{3}}+\floor{\frac{k}{3}}$ and $e^0_{2k+1}+a_{2k+1}= \floor{\frac{2k+1}{3}}$. In that case $d^0_{2(k+1)}+d^0_{2k}=e^0_{2k+1}+a_{2k+1}$.
\item If ${2k\smin 1\equiv 1 [3]}$, then $d^0_{2(k+1)}+d^0_{2k}=\floor{\frac{k+1}{3}}+1+\floor{\frac{k-1}{3}}$ and $e^0_{2k+1}+a_{2k+1}= \floor{\frac{2k+2}{3}}$. In that case $d^0_{2(k+1)}+d^0_{2k}=e^0_{2k+1}+a_{2k+1}$.
\item If ${2k\smin 1\equiv 2 [3]}$, then $d^0_{2(k+1)}+d^0_{2k}=1+\floor{\frac{k}{3}}+1+\floor{\frac{k-1}{3}}$ and $e^0_{2k+1}+a_{2k+1}= 1+\floor{\frac{2k+1}{3}}$. In that case $d^0_{2(k+1)}+d^0_{2k}=e^0_{2k+1}+a_{2k+1}$.
\end{itemize}
The same can be done to prove relation (\ref{lemme_r2}). The relation (\ref{lemme_r3}) is a direct consequence of the definition of $a_{2k+1}$.
\end{proof}

\begin{deff}
Take $m \in \Z_{\geq 1}$. For $k \in \llbracket 1,m \rrbracket$, define the slice $k$ of $\lb_m$ to be the subset of the Young diagram of $\lb_m$: ${\{(i,j) \in \mathcal{Y}(\lb_m) | i+j=k \smin 1\}}$. Note that if we remove the slice $m$ to $\lb_m$ we obtain $\lb_{m\smin 1}$.
\end{deff}

\noindent We are now able to finish the proof of the Theorem \ref{thm_dim_0}.
\begin{prop}
\label{prop_rel_res}
For each $m \in \Z_{\geq 0}$:
\begin{equation*}
\langle \mathrm{Res}_{\tilde{E}_6}(\lb_m),\mathrm{Res}_{\tilde{E}_6}(\lb_m) \rangle = {\mathrm{Res}_{\tilde{E}_6}(\lb_m)}_{\chi_0}.
\end{equation*}
\end{prop}
\begin{proof}
Let us proceed by induction on $m$. The proposition is clear for $m=0$. Due to the definition of $\Res_{\tilde{E}_6}$, we need to make a disjunction elimination on the parity of $m$.
\begin{itemize}
\item If $m=2k$, then by Definition \ref{res_e6}:

\begin{center}
\begin{align*}
\langle \mathrm{Res}_{\tilde{E}_6}(\lb_m),\mathrm{Res}_{\tilde{E}_6}(\lb_m) \rangle = &\langle \mathrm{Res}_{\tilde{E}_6}(\lb_{m\smin 1}),\mathrm{Res}_{\tilde{E}_6}(\lb_{m\smin 1}) \rangle +\\& \langle \mathrm{Res}_{\tilde{E}_6}(\lb_{m\smin 1}), \beta_m \rangle +\\& \langle \beta_m,  \mathrm{Res}_{\tilde{E}_6}(\lb_{m\smin 1}) \rangle +\\& {(d_m^0)}^2+2{(d_m)}^2.
\end{align*}
\end{center}
In the interests of readability, let us denote $\mathrm{Res}_{\tilde{E}_6}(\lb_{m})$ by $\mathrm{R}(m)$. By the induction hypothesis, we have that $\langle \mathrm{R}(m\smin 1),\mathrm{R}(m\smin 1) \rangle={\mathrm{R}(m\smin 1)}_{\chi_0}$. By Definition \ref{res_e6}, we know that ${\mathrm{R}(m\smin 1)}_{\chi_0}={\mathrm{R}(m)}_{\chi_0}$, since $m$ is even. It is then enough to prove that:
\begin{equation}
\label{goal_eq}
{\langle \mathrm{R}(m\smin 1), \beta_m \rangle + \langle \beta_m,  \mathrm{R}(m\smin 1) \rangle + {(d_m^0)}^2+2{(d_m)}^2=0}.
\end{equation}

\noindent By Definition \ref{res_e6}, we have that
\begin{equation*}
\begin{cases}
\mathrm{R}(m\smin 1)_{\chi_0}=\sum_{j=0}^{k\smin 1}{e_{2j+1}^0}\\
\mathrm{R}(m\smin 1)_{\Chi}=\sum_{j=0}^{k\smin 1}{a_{2j+1}}\\
\mathrm{R}(m\smin 1)_{\chi_{\std}}=\sum_{j=0}^{k\smin 1}{d_{2j}^0}.
\end{cases}
\end{equation*}

\noindent  Thanks to equation (\ref{lemme_r1}), we obtain that

\begin{equation*}
\sum_{j=0}^{k\smin 1}{d_{2(j+1)}^0}+\sum_{j=0}^{k\smin 1}{d_{2j}^0}=\sum_{j=0}^{k\smin 1}{e_{2j+1}^0}+\sum_{j=0}^{k\smin 1}{a_{2j+1}}.
\end{equation*}

\noindent Since $d^0_j=0$ for $j=0$, the previous equality gives us
\begin{equation}
\label{eqdm0}
d_m^0=\mathrm{R}(m\smin 1)_{\chi_0}+\mathrm{R}(m\smin 1)_{\Chi}- 2\mathrm{R}(m\smin 1)_{\chi_{\std}}.
\end{equation}
Moreover, thanks to Proposition \ref{res_useful}, we have that the number of boxes that lie in the odd slices between the slice $1$ and $m\smin 1$ is equal to
\begin{equation*}
\mathrm{R}(m\smin 1)_{\chi_0}+ 2\mathrm{R}(m\smin 1)_{\psi} + 3\mathrm{R}(m\smin 1)_{\Chi}.
\end{equation*}
Thus:
\begin{equation*}
\mathrm{R}(m\smin 1)_{\chi_0}+ 2\mathrm{R}(m\smin 1)_{\psi} + 3\mathrm{R}(m\smin 1)_{\Chi}=k^2.
\end{equation*}
In the same way, the number of boxes that lie in the even slices between the slice $1$ and $m\smin 1$ is equal to $2\mathrm{R}(m\smin 1)_{\chi_{\std}}+ 4\mathrm{R}(m\smin 1)_{\psi\chi_{\std}}$. Thus:
\begin{equation*}
2\mathrm{R}(m\smin 1)_{\chi_{\std}}+ 4\mathrm{R}(m\smin 1)_{\psi\chi_{\std}} = k(k\smin 1).
\end{equation*}
From there one has these two relations:
\begin{equation*}
\begin{cases}
d_m^0&=\mathrm{R}(m\smin 1)_{\chi_0}+\mathrm{R}(m\smin 1)_{\Chi}- 2\mathrm{R}(m\smin 1)_{\chi_{\std}}\\
k&=\mathrm{R}(m\smin 1)_{\chi_0}+ 2\mathrm{R}(m\smin 1)_{\psi} + 3\mathrm{R}(m\smin 1)_{\Chi} - \left(2\mathrm{R}(m\smin 1)_{\chi_{\std}}+ 4\mathrm{R}(m\smin 1)_{\psi\chi_{\std}}\right)
\end{cases}.
\end{equation*}
This implies that
\begin{equation*}
k-d_m^0=2\left(\mathrm{R}(m\smin 1)_{\psi} + \mathrm{R}(m\smin 1)_{\Chi} - 2\mathrm{R}(m\smin 1)_{\psi\chi_{\std}}\right).
\end{equation*}
Now since $m=2k$ and $d_m=\frac{m-2d_m^0}{4}$, we obtain
\begin{equation}
\label{eqdm}
d_m=\mathrm{R}(m\smin 1)_{\psi} + \mathrm{R}(m\smin 1)_{\Chi} - 2\mathrm{R}(m\smin 1)_{\psi\chi_{\std}}.
\end{equation}

\noindent Recall that since $m$ is even, $\beta_m=d_m^0\alpha_{\chi_{\std}}+d_m\alpha_{\psi\chi_{\std}}+d_m \alpha_{\psi^2\chi_{\std}}$. By construction of the Euler form, we have
\begin{equation*}
\langle \mathrm{R}(m\smin 1), \beta_m \rangle+\langle \beta_m,  \mathrm{R}(m\smin 1) \rangle= \boldsymbol{(}\mathrm{R}(m\smin 1), \beta_m\boldsymbol{)}
\end{equation*}
were $\boldsymbol{(},\boldsymbol{)}$ denotes the nondegenerate bilinear form on $\mathfrak{h}^*_{\Gamma}$ \cite[§2.1]{kac90}. Using the McKay graph of type $\tilde{E}_6$, we deduce that
\begin{center}
\begin{align*}
\boldsymbol{(}\mathrm{R}(m\smin 1), \beta_m\boldsymbol{)}=
2d_m^0\mathrm{R}(m\smin 1)_{\chi_{\std}}&+4d_m\mathrm{R}(m\smin 1)_{\psi\chi_{\std}}-d_m^0\left(\mathrm{R}(m\smin 1)_{\chi_0}+\mathrm{R}(m\smin 1)_{\Chi} \right)\\
&-2d_m\left(\mathrm{R}(m\smin 1)_{\Chi}+\mathrm{R}(m\smin 1)_{\psi}\right).
\end{align*}
\end{center}
\noindent Rearranging the right-hand side of the previous equation, it is equal to
\begin{equation*}
d_m^0\Bigl(2\mathrm{R}(m\smin 1)_{\chi_{\std}}-\mathrm{R}(m\smin 1)_{\chi_0}-\mathrm{R}(m\smin 1)_{\Chi} \Bigl)+2d_m\Bigl(2\mathrm{R}(m\smin 1)_{\psi\chi_{\std}}-\mathrm{R}(m\smin 1)_{\Chi}-\mathrm{R}(m\smin 1)_{\psi}\Bigl).
\end{equation*}

\noindent We recognise the expressions of $d_m^0$ and $d_m$ obtained in (\ref{eqdm0}) and (\ref{eqdm}). This gives us that
\begin{equation*}
\boldsymbol{(}\mathrm{R}(m\smin 1), \beta_m\boldsymbol{)}=-{(d_m^0)}^2-2{(d_m)}^2.
\end{equation*}

\noindent This gives the desired equality (\ref{goal_eq}) and concludes the proof when $m$ is even.
\item Let us suppose now that $m=2k+1$ is odd. It is then enough to prove that
\begin{equation*}
 \langle \mathrm{R}(m\smin 1), \beta_m \rangle + \langle \beta_m,  \mathrm{R}(m\smin 1) \rangle + {e_m^0}^2+2{e_m}^2+{a_m}^2=e_m^0.
\end{equation*}
Thanks to relations (\ref{lemme_r2}) and (\ref{lemme_r3}), we firstly have that
\begin{equation}
\label{p1_odd}
2\mathrm{R}(m \smin 1)_{\chi_0}+e^0_m-\mathrm{R}(m\smin 1)_{\psi\chi_{\std}}=1.
\end{equation}
Secondly, we have that
\begin{equation}
\label{p2_odd}
4\mathrm{R}(m\smin 1)_{\psi}+2e_{m}-2\mathrm{R}(m\smin 1)_{\psi\chi_{\std}}=0.
\end{equation}
Thirdly, we have that
\begin{equation}
\label{p3_odd}
2\mathrm{R}(m \smin 1)_{\Chi}+a_m-\mathrm{R}(m \smin 1)_{\chi_{\std}}-2\mathrm{R}(m \smin 1)_{\psi\chi_{\std}}=0.
\end{equation}
The equations (\ref{p1_odd}), (\ref{p2_odd}) and (\ref{p3_odd}) give the desired result when $m$ is odd and concludes the proof of the proposition and also of Theorem \ref{thm_dim_0}.
\end{itemize}
\end{proof}

\begin{rmq}
Theorem \ref{thm_dim_0}, implies that if $\Gamma$ is a finite subgroup of $\SL_2(\C)$ isomorphic to the binary octahedral group (of type $\tilde{E}_7$) or if $\Gamma$ is a finite subgroup of $\SL_2(\C)$ isomorphic to the binary icosahedral group (of type $\tilde{E_8}$), then all irreducible components of $\Hk_{n}^{\Gamma}$ containing a $\T_1$-fixed point are of dimension $0$ since these two finite groups contain a subgroup of type $\tilde{E}_6$.
\end{rmq}

\printbibliography

\Addresses
\end{document}